\documentclass[a4paper,12pt]{article}
\usepackage[margin=1in]{geometry}  

\usepackage{graphicx}              
\usepackage{amsmath}
\usepackage{amscd}               
\usepackage{amsfonts} 
\usepackage{amssymb}             
\usepackage{amsthm}                
\usepackage{epsfig}

\theoremstyle{plain}
\newtheorem{thm}{Theorem}[section]

\newtheorem{theorem}{Theorem}[section]

\newtheorem{lemma}[thm]{Lemma}
\newtheorem{corollary}[thm]{Corollary}
\theoremstyle{definition}

\newtheorem{definition}{Definition}
\theoremstyle{remark}

\newtheorem{remark}{Remark}

\DeclareMathOperator{\id}{\rm{Id}}

\newcommand{\R}{\mathbb{R}}      

\newcommand{\N}{\mathbb{N}}

\newcommand{\cL}{\mathcal{L}}
\newcommand{\cO}{\mathcal{O}}		 
\newcommand{\cN}{\mathcal{N}}		 
\newcommand{\cP}{\mathcal{P}}		 
\newcommand{\E}{\mathbb{E}}     
\newcommand{\PP}{\mathbb{P}}     

\newcommand{\var}{\text{var}}

\begin{document}
\title{Asymptotic independence of Spearman's uniform metric with other metrics}
\author{Yunjiang Jiang}
\maketitle
\section{Introduction}
	Symmetric groups have been extensively studied in mathematics for over 3 centuries with a tradition going back to Lagrange. Many fundamental applications now exist in almost every branch of mathematics. In recent decades, statisticians and probabilists made attempts to understand their properties as finite probability spaces. Many of the important quantities on the symmetric groups can be extended to a metric structure. 
	
	The ranking of n data points, for example, can be construed as a permutation in $S_n$. One can thus assess the difference between two sets of data by defining metrics on the symmetric group. If the two sets of data are components of n samples of a bivariate variable, then one would like the distance function to be invariant from one side, since it should be invariant under relabeling of the samples. In the latter case, the metric serves as a non-parametric correlation function. 
	
	Another theoretical use of metrics on the symmetric groups is to understand the convergence rate of shuffling models. Here to confuse things further, the rate of convergence is measured in terms of another type of metric, which instead of on the group itself lives on the space of all probability measures on $S_n$. Typically one considers total variation distance or in the presence of an underlying metric on the probability space, transportation distance associated with the underlying metric. In the former case, one particular statistic, the Hamming distance, has proved quite useful in getting tight lower bound in many models. The latter is less well studied but conceivably is most naturally attacked via their underlying metrics as test statistics.
	
	As in principal component analysis, one would like to remove redundant information by identifying metrics that are statistically dependent on one another. It is thus surprising that many metrics that are dependent for finite $n$ become asymptotically independent (in a sense made precise later) as $n$ approaches infinity. This suggests that for large samples, these metrics give truly distinct measures of correlation. On the other extreme, one also find certain pairs of metrics that are asymptotically perfectly correlated. In this paper we examine some of the most popular metrics on the symmetric groups and their asymptotic statistical relations. 

\section{Metrics on finite groups}
In general, a metric on a set $G$ of points is any bivariate positive-valued function $\rho$ that satisfies the following three properties for all $x,y,z \in G$:
\begin{enumerate}
\item Nondegeneracy: $\rho(x,y) = 0$ if and only if $x =y$,
\item Symmetry: $\rho(x,y) = \rho(y,x)$, and
\item Triangle inequality: $\rho(x,z) \le \rho(x,y) + \rho(y,z)$.
\end{enumerate}
	If the underlying set happens to be a group, then additional structures can be imposed. 
	
	We say that $\rho$ is left-invariant or right-invariant if respectively, 
	\begin{align*}
	\rho(ax,ay) = \rho(x,y)
	\end{align*}
	or 
	\begin{align*}
	\rho(xa,ya) = \rho(x,y).
	\end{align*}
 It is said to be bi-invariant if it is both left and right invariant. 
 
 A left invariant metric $\rho$ (and similarly a right-invariant one) can be characterized by the univariate function $f(x) = \rho(\id, x)$, so that
 \begin{align*}
 \rho(x,y) = f(x^{-1} y).
 \end{align*}
In this setting, the three properties of a metric can be rephrased as following:
\begin{enumerate}
\item Nondegeneracy: $f(x) = 0$ if and only if $x = \id$,
\item Symmetry: $f(x) = f(x^{-1})$ for all $x \in G$, and
\item Triangle inequality: $f(x) \le f(y) + f(x^{-1} y)$ or $f(x) \le f(xy) + f(y)$, for all $x,y$.
\end{enumerate}
It is a challenge to give statistical meaning of the triangle inequality in applications.

 For a bi-invariant metric, its characterizing univariate function is also a class function,
 \begin{align*}
  f(a x a^{-1}) = f(x),
  \end{align*}
  which means that it projects to a function on the set of conjugacy classes of $G$. When $G = S_n$, a symmetric group, the conjugacy class to which a permutation belongs is uniquely determined by its cycle structure. Therefore the set of conjugacy classes is parametrized by $\cP_n$, the set of partitions of $n$. Whenever we talk about a class function $f$ on $S_n$, it will simultaneously denote the function on $S_n$ and the projected function on $\cP_n$, whenever there is no confusion.

\section{Asymptotic independence}
Before we study statistical dependence on the symmetric group, it is necessary to have a clear definition of asymptotic independence. Given two sequences of real-valued random variables $X_i,Y_i$, $i = 1,2,\ldots$, we distinguish between two modes of asymptotic independence:
\begin{definition}
$X_i$ and $Y_i$ are said to be weakly asymptotically independent if for all bounded continuous functions $f,g$, 
\begin{align*}
\lim_{i \to \infty} \E [f(X_i) g(Y_i)] = \lim_i \E f(X_i) \lim_i \E g(Y_i).
\end{align*}
They are said to be asymptotically independent in moment if all of their moments exist and for all $j,k \in \N$,
\begin{align*}
\lim_i \E X_i^j Y_i^k = \lim_i \E X_i^j \lim_i \E Y_i^k.
\end{align*}
\end{definition}
Notice that the definitions above imply that both sequences have weak limits and weak limits with all moments finite respectively. Moment independence clearly implies weak independence. The other direction is also true, provided that $X_i$ and $Y_i$ both converge in moments:
\begin{proof}
Consider the bounded function $f_{N,k}(x) = |x|^k 1_{\{x < N\}}$. Then for fixed $N$, 
\begin{align*}
\lim_i \E[f_{N,j}(X_i) g_{N,k}(Y_i)] = \lim_i \E[f_{N,j}(X_i)] \lim_i \E[g_{N,k}(Y_i)]
\end{align*}
Taking $N \to \infty$ on both sides, interchanging limits on the left and using continuity of the function $f(x,y) = xy$ on the right, we obtain
\begin{align*}
\lim_i \E[X_i^j Y_i^k] = \lim_i \E [X_i^j] \lim_i \E [Y_i^j]
\end{align*}
\end{proof}
Since we are dealing mostly with random variables that have moments of all orders when scaled appropriately, we can talk about the two modes of asymptotic independence interchangeably. Unlike modes of convergence, asymptotic independence is essentially a weak phenomenon, hence there are no stronger versions such as in $\cL^1$ or point-wise.

	Often one needs to know roughly how large an $n$ suffices for the statistical implications of two functions to decouple. This requires a distance function on the space of all probability measures on $\R^2$. The one commonly used in Berry-Esseen type estimate is Kolmogorov distance, defined by
	\begin{align*}
	d(\mu,\nu) = \sup_x |\mu((-\infty,x]) - \nu((-\infty,x])|.
	\end{align*}
	Other distances, such as total variation, is less appropriate due to the presence of atoms in one of the measures. 
	
	 The definition below is based on Kolmogorov's distance:
	\begin{definition}
 	Two sequences of random variables $X_i,Y_i$, $i =1,2,\ldots$, are said to be asymptotically independent at the rate $r(i)$ under the if 
 	\begin{align*}
 	sup_{x, y \in \R} |\PP[X_i \le x, Y_i \le y] - \PP[X_i \le x] \PP[Y_i \le y] | = O(r(i)).
 	\end{align*}
	\end{definition}
	
\section{Spearman's footrule, rho, and generalizations}

Viewing each permutation as a bijection from $[n]$ to itself, one can easily come up with "natural" functions on $S_n$. The most obvious ones are of course the coordinate functions, which however is not a good candidate for building correlation kernels. The sum of all the components is of course trivially constant, so one has to look harder. Spearman \cite{Sp} gave the following two analogues of the $\cL^1$ and $\cL^2$ norms, known as the footrule and rho respectively:
\begin{align*}
\rho_1(\sigma) &= \sum_{i=1}^n |\sigma(i) - i|\\
\rho_2(\sigma) &= \sum_{i=1}^n	(\sigma(i) - i)^2
\end{align*}
Note that Spearman's rho differs from the correct $\cL^2$ norm by a square root. Since it is invertible, this extra facade has no bearing on independence, and it makes calculations much easier. One can also define the analogue of $\cL^p$ norms, $\rho_p$, as above. The combinatorial central limit theorem implies that they all converge weakly to the standard normal when scaled and recentered, and that the error term can be controlled by a Berry-Esseen type estimate \cite{Bolt}. Thus to understand their limits, it suffices to compute the means and variances. This can be effectively done by writing the sum out explicitly and split into cases where the indices are the same or different. One can use the same method to compute correlations of $\rho_p$ and $\rho_q$. 

	Observe now that any finite linear combination of the $\rho$'s is in the domain of normal convergence, also by combinatorial CLT, one obtains the full joint distribution of all the $\rho_p$'s, since they are jointly Gaussian in the limit.

\section{Bi-invariant metrics and other metrics}
	Recall that a metric $\rho$ on a group $G$ is called bi-invariant if $\rho(agb,ahb) = \rho(g,h)$ for all $a,b,g,h \in G$. Such metrics arise naturally in the context of subjective ranking. Suppose two wine connossieurs are to rank $n$ bottles of wines, and we want to assess how correlated their tastes are. Then not only are the order in which the wines are presented irrelevant for the analysis, so are the actual rankings themselves. What matters is whether or not the two connossieurs assign the same relative value to each bottle. Thus if their rankings are presented as two permutations, and their valuation difference is measured by a metric, then the metric should be bi-invariant. 
	
	More generally, one could consider data valued in an unordered set. A good example is given by the election of cabinet members. Suppose voters (or say the president and the congress) are to assign $n$ office positions to $n$ candidates already cleared for hiring, it would be useful to know how much the voters agree or differ on how to match. Since the cabinet positions are supposed to be non-hierarchical (unlike the president and the vice-president), one expects the measure of discordance to be invariant both under the reshuffling of the candidates and of the positions sought. Of course in politics there are typically more candidates than available positions. Private companies on the other hand tend to preserve existing work force by matching employees with jobs, such as during intern assignment.
  
  Furthermore, subjective assignments in the ranked setting often clash with objective scales that should correlate strongly with the ranking. In the wine tasting example, the $n$ bottles might be made in different years. One could on the one hand tabulate each judge's rankings against the production years, and compare the resulting permutations using some left-invariant metric such as Spearman's rho, and on the other hand compare their rankings alone using some bi-invariant metric, such as the Hamming distance. It would be useful to know whether the two measures of disarray provide more information than one measure alone. The results below give affirmative answers, at least when $n$ is sufficiently large.

	\begin{lemma}
	Let $\lambda^n \vdash n$ be given for each $n$ such that the number of parts in $\lambda$, $t:=n(\lambda) \le n^{\epsilon}$ for $\epsilon < 1/6$, then 
	\begin{align*}
	\lim_n \E[f(\bar{\rho}(\sigma)) | \lambda(\sigma) = \lambda^n] = \lim_n \E[f(\bar{\rho})(\sigma)]
	\end{align*}
	where $f$ denotes any polynomial growth continuous function and $\rho$ denotes any of the following 
	\begin{itemize}
	\item Spearman's $\rho_q$ for $1 \le q < \infty$,
	\item Kendall's tau
	\item length of the longest increasing/decreasing subsequence
	\end{itemize}
	and $\bar{\rho}$ stands for affine normalization of $\rho$ to have mean 0 and variance 1.
	\end{lemma}
	\begin{remark}
	We can easily generalize the result to other functions, but instead choose to focus on the well-known ones for simplicity.
	\end{remark}
	\begin{proof}
	Write $\sigma \in S_n$ in the following record cycle form:
	\begin{align*}
	\sigma = (a_{11} \ldots a_{1 s_1})\ldots (a_{t1} \ldots a_{ts_t})
	\end{align*}
	with the property that $a_{i1} \ge a_{ij}$ for all $i,j$, and $a_{11} < a_{21} < \ldots < a_{t1}$. The record map $r: S_n \to S_n$ is defined by
	\begin{align*}
	r(\sigma)( \sum_{i=1}^{j-1} t_i + k) = a_{jk}. 
	\end{align*}
	In words, we remove all the brackets in the record cycle representation of $\sigma$ and treat the resulting sequence as the second row of a permutation written in 2-line form. 
	
	Let $S_n^\lambda$ denote the set of permutations with cycle structure $\lambda$. Given $\lambda \vdash n$, consider the following map
	\begin{align*}
	Y^\lambda: S_n &\to S^\lambda_n \\
	\sigma &\mapsto (\sigma(1) \ldots \sigma(\lambda_1))(\sigma(\lambda_1 + 1) \ldots \sigma(\lambda_1 + \lambda_2)) \ldots (\sigma(\lambda_1 + \ldots + \lambda_{t-1} + 1) \ldots \sigma(n)).
	\end{align*}

	This map pushes the uniform measure on $S_n$ onto the uniform measure on $S^\lambda_n$. So $M^\lambda := Y^\lambda \circ r: S_n \to S_n^\lambda$ also pushes the uniform measure to uniform. Thus 
	\begin{align} \label{uniform push}
	\E[f(\rho_q(\sigma)) | \lambda(\sigma) = \lambda] = \E[f(\rho_q(M^\lambda(\sigma)))].
	\end{align}
	Furthermore $Y^\lambda$ changes at most $n^\epsilon$ coordinates of $\sigma$, by the condition on $\lambda$. Therefore $M^\lambda$ changes at most $n^\epsilon + n(\lambda(\sigma))$ coordinates. By the central limit theorem for the number of cycles, $\PP[n(\lambda(\sigma)) > k \log n]  = O(n^{1-k})$.  
	
	In the case of Spearman's $\rho_q$ function, the standard deviation is of order $n^{q + 1/2}$, and changing $O(n^\epsilon)$ coordinates alters its value by $O(n^{q + \epsilon}) << n^{q + 1/2}$, hence we have the following convergence under the uniform measure on $S_n$:
	\begin{align*}
 \lim_{n \to \infty} \PP[|(\rho_q(M^\lambda(\sigma)) - \rho_q(\sigma))/\sqrt{\var \rho_q}| > \epsilon] = 0,
 \end{align*}
 for all $\epsilon > 0$. In other words
 \begin{align*}
 \lim_n \bar{\rho}_q(M^\lambda(\sigma)) - \bar{\rho}_q(\sigma) = 0,
 \end{align*}
 in probability.
 
Combining with \eqref{uniform push}, we have
\begin{align*}
\lim_n \E[f(\bar{\rho}_q(\sigma)) | \lambda(\sigma) = \lambda] - \E[f(\bar{\rho}_q(\sigma))] &=  \lim_n \E[f(\bar{\rho}_q(M^\lambda(\sigma))) - f(\bar{\rho}_q(\sigma))] ( 1 - \PP[n(\lambda) \le k \log n])\\
&+ \max_\sigma \bar{\rho}_q(\sigma) \PP[n(\lambda) > k \log n] \\
&= 0,
\end{align*}
for $k > 2$, by dominated convergence theorem, and the fact that 
\begin{align*}
\max_\sigma \bar{\rho}_q(\sigma) &= O(n^{q+1}) / \sqrt{\var \rho_q}\\
&\le  O(n^{1/2}).
\end{align*}

Kendall's tau $\tau(\sigma) := \sum_{i < j} 1_{\{\sigma(i) > \sigma(j)\}}$ has variance of order $n^3$, whereas the change of one coordinate value would affect $O(n)$ terms in the sum, each of which has contribution $O(1)$. Therefore $|\tau(M^\lambda(\sigma)) - \tau(\sigma)| = O(n^{1 + \epsilon}) = o(\sqrt{\var \tau})$. Similarly, the length of the longest increasing sequence $U(\sigma)$ has variance of order $n^{1/6}$ whereas changing one coordinate in $\sigma$ would change $U$ by at most $2$. Thus the same argument for Spearman's rho functions apply to the latter two cases as well.
\end{proof}

\begin{corollary}
Any sequence of class functions $f_n$ on $S_n$ with a weak limit is asymptotically independent of all the functions listed below, with the second column giving upper bound on rates of convergence:
\begin{itemize}
\item The normalized Spearman's rho's, $\bar{\rho}_q$, $1 \le q < \infty$; $r(n) = \log n/ n^{1/2}$.
\item The normalized Kendall's tau; $r(n) = \log n / n^{1/2}$.
\item The normalized Ulam's statistic (longest increasing subsequence); $r(n) = \log n / n^{1/6}$.
\end{itemize}
\end{corollary}

\begin{proof}
It is well-known that the number of cycles $n(\sigma)$ in a uniformly chosen permutation satisfies the central limit theorem with mean $\log n$ and variance $\log n$. Thus for any $k > 1$,
\begin{align*}
\PP[n(\sigma) > k \log n] = O(Erf(-(k-1) \sqrt{\log n})) = O(n^{1-k}).
\end{align*}
 Next using the estimate of the previous lemma, we have for $n(\lambda) \le k \log n$, and $k > 2$, 
 \begin{align*}
 \PP[\bar{\rho}_q(\sigma) < a - \frac{k \log n}{n^{1/2}} | \lambda(\sigma)  =\lambda] \le \PP[\bar{\rho}_q(\sigma) < a] \le \PP[\bar{\rho}_q(\sigma) < a + \frac{ k\log n}{n^{1/2}}].
 \end{align*}
 Using the fact that $\bar{\rho}_q$ weakly converges to a standard normal, and in fact satisfies a Berry-Esseen's error estimate of order $n^{-1/2}$,
 \begin{align*}
 |\PP[\bar{\rho}_q(\sigma) < a | \lambda(\sigma) = \lambda] - \PP[\bar{\rho}_q(\sigma) < a] | = \frac{C\log n}{n^{1/2}},
 \end{align*}
for some universal $C$. 

 Next since $f_n$ are class functions, they project to functions on the set $\cP_n$ of partitions of $n$. Summing over all $\lambda$ with $n(\lambda) \le k \log n$, we obtain
 \begin{align*}
 \PP[\bar{\rho}_q(\sigma) < a, f_n(\sigma) < b] &= \sum_{\lambda: n(\lambda) \le k \log n} \PP[\bar{\rho}_q(\sigma) 1_{\{f_n(\lambda) < b\}}| \lambda(\sigma) = \lambda] \PP[S_n^\lambda]  + O(n^{1-k})\\
 &= \sum_{\lambda: n(\lambda) \le k \log n} (\PP[\bar{\rho}_q(\sigma) < a] + \epsilon_\lambda) \PP[S_n^\lambda] + O(n^{1-k}),
 \end{align*}
 where $|\epsilon_\lambda| \le C \log n / n^{1/2}$, save a set of $\lambda$'s of probability at most $O(n^{-1/2})$. Finally,
 \begin{align*}
 \PP[\bar{\rho}_q(\sigma) < a, f_n(\sigma) < b] - \PP[\bar{\rho}_q(\sigma) < a] \PP[f_n(\sigma) < b] &=
 \sum_{\lambda: n(\lambda) \le k \log n} \epsilon_\lambda \PP[S_n^\lambda] + O(n^{-1/2})+ O(n^{1-k})\\
 &\le O(\log n / n^{1/2}).
 \end{align*} 
The proof of the other two pairs (Kendall's tau and Ulam's statistic) follow the same argument and is omitted.
\end{proof}

\section{An application}
An interesting spin-off of the Spearman-type statistics is called the oscillation of permutations, defined as
\begin{align*}
\rho^{(1)}_q(\sigma) = \sum_{i=1}^n |\sigma(i+1) - \sigma(i)|^q.
\end{align*}
The notation $\rho^{(1)}$ suggests the analogy with Sobolev norms in classical analysis. It was first shown in \cite{Bai1} that a central limit theorem can be proved about such statistics, by moment method. Later in \cite{Chao} a tight Berry-Esseen error estimate is derived using Stein's method. More precisely, they showed that
\begin{align*}
\sup_x |\PP[\rho^{(1)}_1(\sigma) < x] - Erf(x)| = O(n^{-1/2}).
\end{align*}

Consider now an $n$-cycle $\tau$ derived from $\sigma$ by the following recipe:
\begin{align*}
\tau^{\circ k} (i_0) = \sigma(k) 
\end{align*}
for some fixed starting index $i_0 \in [n]$. Here $\tau^{\circ k} = \tau \circ \ldots \circ \tau$ for $k$ times. Then it is easy to verify that
\begin{align*}
\rho^{(1)}_q(\sigma) = \rho_q(\tau).
\end{align*}
We denote the map $\sigma \mapsto \tau$ by $r_{i_0}: S_n \to S_n^{(n)}$. Then each one is measure-preserving, and using the Hoeffding combinatorial central limit theorem\cite{Bolt} with Bolthausen error term for $\rho_q$, we obtain
\begin{align*}
\PP[\bar{\rho}_q^{(1)} (\sigma) < x] &= \PP[\bar{\rho}_q(r_{i_0}(\sigma)) < x]\\
&= \PP[\bar{\rho}_q(\sigma) < x | \lambda(\sigma) = (n)]\\
&= Erf(x) + O( \frac{\log n}{n^{1/2}})
\end{align*}
So we are off by a factor of $\log n$ in the error term, but the proof is significantly shorter than using Stein's method as in \cite{Chao}. Note also that we get the same normalization affine map for $\rho^{(1)}_q$ and $\rho_q$.

  Using the same method, one could get central limit theorem with near-tight error bound for all the analogues of Sobolev norm, $\rho^{(p)}_q$, provided one cna establish CLT for random diagonal of the form $\sum_{i=1}^n a_{i, \sigma(i), \ldots, \sigma^k(i)}$ in higher dimensional arrays. A less ambitious proposition is to show that the skip-$2$ (or in general skip-$k$) Sobolev $\rho_q^{(1)}$ norms are asymptotically Gaussian:
  \begin{align*}
  \rho_{q,2}^{(1)} (\sigma) := \sum_{i=1}^n |\sigma(i+2) - \sigma(i)|^q,
  \end{align*}
  where again the summation indices are taken modulo $n$. Using the same conditioning argument as before, one needs CLT for the following variant of the Spearman's $\rho$ metrics:
  \begin{align*}
	\rho_{q,2}(\sigma) := \sum_{i=1}^n | \sigma^2(i) - i|^q.
	\end{align*}
	
	\begin{proof}
	Observe that the set $\{\sigma^2: \sigma \in S_n\}$ consists of permutations $\tau$ with $\alpha_{2k}(\tau) \equiv 0$ mod $2$, for all $k \le n$, where $\alpha_j(\sigma)$ denotes the number of $j$-cycles in $\sigma$.  Define a bijection $\tau$ on $S_n$, $\tau: \sigma \mapsto \sigma'$, as follows:
	given $\sigma \in S_n$, and a $j$-cycle $\gamma = (a_1 a_2 \ldots a_j)$ in $\sigma$ arranged so that $a_1 = \max a_i$ (i.e. in record form), we will let $\gamma'$ be a $j$-cycle in the image $\sigma'$:
\begin{enumerate}
\item 
if $j = 2k$, and say $\gamma = (a_1 a_2 \ldots a_{2k})$, let $\gamma' = (a_1 a_3 \ldots a_{2k-1} a_2 a_4 \ldots a_{2k})$. 
\item
if $j = 2k+1$, let $\gamma'= (a_1 a_3 \ldots a_{2k+1} a_2 \ldots a_{2k})$.
\end{enumerate}
The resulting $\sigma'$ has the same cycle structure as $\sigma$ and the map described above is a conjugacy class preserving bijection. Note that it is necessary to standardize each cycle to the record form (or by some other convention) in order for $\tau$ to be bijective. 

 Furthermore, when we post-compose $\tau$ with the appropriate bracket inserting operation $\beta$, we recover the familiar map $\sigma \mapsto \sigma^2$. Here $\beta(\sigma)$ breaks each even cycle in $\sigma$ arranged with largest element first into two equal smaller ones, at a position determined by the action of the map $\sigma \mapsto \sigma^2$ on that cycle; it leaves the odd cycles untouched. For example, if $\gamma = (612345)$, $\tau(\gamma) = (624135)$ and $\beta \circ \tau(\sigma) = (624)(135) = (624)(513)$. 

 Now since with high probability, $\sigma$ sampled from the uniform measure has fewer than $k \log n$ cycles, for any $k \ge 3$, $\beta$ changes the image of $\tau$ by at most $k \log n$ coordinates with high probability, so the same perturbation argument presented in the previous sections show that for all $\epsilon > 0$,
 \begin{align*}
 \lim_{n \to \infty} \PP[ |\rho_q(\tau(\sigma)) - \rho_q(\sigma^2)| > \epsilon \sqrt{\var(\rho_q)}] = 0, 
\end{align*}
and that almost surely,
\begin{align*}
\lim_{n \to \infty} [\rho_q (\tau(\sigma)) - \rho_q(\sigma^2)] / \sqrt{\var(\rho_q)} = 0. 
\end{align*}
Since $\tau$ is a (measure-preserving) bijection, the random variable $\rho_q(\sigma^2)$ is also asymptotically Gaussian, whose cumulative distribution function differs from the normal one by $O( \log n /\sqrt{n})$ as before, as a consequence of the Bolthausen error bound.
\end{proof} 

Finally we observe that $\rho_2^{(2)}$ can be written in the following form:
\begin{align*}
\rho_2^{(2)}(\sigma) &:= \sum_{i=1}^n ((\sigma^2(i) - \sigma(i)) - (\sigma(i) - i))^2 \\
&= 2\sum_{i=1}^n (\sigma(i) - i)^2 - 4 \sum_{i=1}^n i \sigma(i) + 2 \sum_{i=1}^n i \sigma^2(i) + C_n\\
&= -8 \sum_{i=1}^n i \sigma(i) + 2 \sum_{i=1}^n i \sigma^2(i) + C'_n
\end{align*}
where $C_n, C'_n$ are some constants. Thus to establish asymptotic normality of $\rho_2^{(2)}(\sigma)$, it suffices to show asymptotic joint normality of $\rho_2(\sigma)$ and $\rho_{2,2}(\sigma)$. This however doesn't seem to follow from elementary considerations. 

\section{Spearman's uniform metric and other metrics}

The natural $\cL^\infty$ generalization of Spearman's footrule and rho metric on $S_n$ is given by the following formula
\begin{align*}
\rho_\infty(\sigma) = \max_{i \le n} |i -\sigma(i)|
\end{align*}
In this note we will study the limiting distribution of $H := n-\rho_\infty$ under the uniform measure on $S_n$ as well as its independence relation with other common metrics on $S_n$.

Throughout the article, lower case $p$ will denote a single point, whereas $P$ will denote a collection of points.

Consider Hammersley's device which samples $n$ iid uniform points in the unit square $[0,1]^2$. For each instance $P=\{p_1,\ldots, p_n\}$ of this point process, define $X(p_i)= X_P(p_i) := $ the number of points to the left of $p_i$, including $p_i$ itself, and $Y(p_i) = Y_P(p_i) := $ the number of points below $p_i$ also self-included. We will also define $X'(p_i)$ (resp. $Y'(p_i)$) to be the number of points strictly to the right of (resp. above) $p_i$. These are almost surely well-defined and distinct for different $p_i$'s. We can associate to $P$ a permutation $\sigma$, defined by $\sigma_P(i) = j$ if there is some $p_k \in P$, with $X(p_k) = i$ and $Y(p_k) = j$. It is easy to see by symmetry that the induced measure on $S_n$ is uniform.

The Spearman's uniform metric $\rho_\infty$ can be expressed in terms of Hammersley's coupling as 
\begin{align*}
\rho_\infty(\sigma_P) = \max_{k \le n} f(p_k).
\end{align*}
where $f(p) = f_P(p) := |X(p) - Y(p)|$. Similarly 
\begin{align*}
H(\sigma_P) = \min_{k \le n} h(p_k)
\end{align*}
where $h(p):= (X(p) + Y'(p)) \wedge (X'(p) + Y(p))$. 
Thus it suffices to study the latter distribution. We will abuse the following notation $H(P)  =H(\sigma_P)$, and similarly for $\rho_\infty$.

Instead of a fixed $n$ ensemble, it is much easier to work with a Poisson point process $\Lambda$ on the unit square with homogeneous rate $\nu \approx n$. We have the following De-Poissonization theorem (see \cite{Johan} Lemma 2.5).
\begin{theorem}
Let $A_n$ be a monotone (increasing or decreasing) sequence with values in $[0,1]$. Define
\begin{align*}
\varphi_A(m) = e^{-m} \sum_{n=0}^\infty \frac{m^n}{n!} A_n.
\end{align*}
Then
\begin{align*}
\varphi_A(N - \sqrt{N \log N}) - \frac{C}{N^2} \le A_N \le \varphi_A(N + \sqrt{N \log N}).
\end{align*}
\end{theorem}
The proof of the theorem relies on approximating the Poisson distribution by the normal distribution near its mean. One can first prove it for $A_n$ increasing, and then consider $B_n = 1 - A_n$ to establish the decreasing case. Since $\varphi_A(N- \sqrt{N \log N})$ approximates the average of $A_m$ for $m$ near $N- \sqrt{N \log N}$, monotonicity yields the first inequality. The second inequality follows a similar argument. 

In many situations, we cannot get exact monotonicity. The following corollary is thus useful.

\begin{corollary} \label{De-Poisson}
If $A_n \in [0,1]$ is a sequence that satisfies $A_n (1 + \frac{c}{n^\delta}) > A_{n+1}$, for some $c$, and $\delta > 1/2$, then 
\begin{align*}
\sum_{n=0}^\infty A_n e^{-\mu} \frac{\mu^n}{n!} - \frac{C \log N}{N^{\delta - \frac{1}{2}}} < A_N < \sum_{n=0}^\infty A_n e^{-\nu} \frac{\nu^n}{n!} + \frac{C \log N}{N^{\delta - \frac{1}{2}}},
\end{align*}
where $\mu = N - \sqrt{N \log N}$ and $\nu = N  + \sqrt{N \log N}$ as before.
\end{corollary}
\begin{proof}
Let $\gamma = N - \sqrt{N (\log N)^2}$ and define $B_n = A_\gamma \prod_{i=\gamma + 1}^n (1 + \frac{c}{i^\delta})^{-1}$ for $n \ge \gamma$ and $B_n = A_n$ otherwise. Then $B_n > B_{n+1}$ for all $n \ge \gamma$. Observe that if $n < \gamma$, 
\begin{align*}
e^{-\mu} \frac{\mu^n}{n!} &=  \exp(-\mu + n \log \mu - n (\log n - 1)) \\
 &\le \exp( n[ 1 + \log \frac{\mu}{n} - \frac{\mu}{n}])\\
 &\le \exp( \cO(n (\frac{\mu}{n} -1)^2) \wedge (\frac{\mu}{n} -1))\\
 &\le \cO(\exp(- (\log N)^2))
 \end{align*}
 because $x- \log(1+x) \le \cO(x^2 \wedge x)$, for $x > 0$.  Therefore the contribution of $\sum_{n=0}^\gamma A_n e^{\mu} \frac{\mu^n}{n!}$ is $\cO(\exp(-(\log N)^2))$ and similarly for $\mu$ replaced by $\nu$. 

 Note that the sequence $B_n \in [0,1]$. Hence by the previous theorem we have
\begin{align*}
\varphi_B(N - \sqrt{N \log N}) - \frac{C}{N^2} \le B_N \le \varphi_B(N + \sqrt{N \log N}) + \frac{C}{N^2}.
\end{align*}
Observe that 
\begin{align*}
\prod_{n =\gamma}^\mu (1 + \frac{c}{n^\delta}) &\le (1+\frac{c}{\gamma^\delta})^{N^{1/2}\log N} \\
&\le \exp( c N^{\frac{1}{2} - \delta} \log N)  \le 1 + \frac{C \log N}{N^{\delta - \frac{1}{2}}}.
\end{align*}
Thus $\frac{\varphi_A(\mu)}{\varphi_B(\mu)} \le 1+ \frac{C\log N}{N^{\delta - \frac{1}{2}}}$ and since $\varphi_B(\mu) \in [0,1]$,
\begin{align*}
 \varphi_A(\mu)-\varphi_B(\mu) \le \frac{C\log N}{N^{\delta - \frac{1}{2}}}.
 \end{align*}
 Similarly, we obtain $A_N - B_N \le \frac{C \log N}{N^{\delta - \frac{1}{2}}}$. This gives the first inequality. The other inequality can be analyzed similarly, with a bigger constant $C$.
\end{proof}

Let $p^* = \arg \max_{p \in P} f(p) = \arg \min_{p \in P} h(p)$, which is almost surely well-defined. Define $g(p) = d(p, \{x=y\})$, where $\{x=y\}$ is the diagonal segment of the unit square, and $d$ stands for Euclidean distance. Also let $\bar{p} = \arg \max_{p \in P} g(p)$. Intuitively, $\bar{p}$ is the point in $P$ whose $x$ and $y$ coordinates are furthest apart. Therefore by law of large number consideration, one would expect $f$ to be maximized at $\bar{p}$.
\begin{lemma}
\begin{align*}
\PP_n[ p^* = \bar{p}] = 1 + o_n(1).
\end{align*}
\end{lemma}
\begin{proof}

Conditional on a typical position of $\bar{p}$, we need to show that with high probability $H(p_i) > H(\bar{p})$ for all $p_i \in P \setminus \{\bar{p}\}$. By a simple calculation, we see that $\frac{\sqrt{2}}{2} - g(\bar{p}) < n^{1/2} \log n$ with high probability. Conditional on $\bar{p}$, the remaining points are distributed iid uniformly on the region $\{g < g(\bar{p})\}$. Furthermore, the trapezoidal region $T= \{g \in [g(\bar{p})-n^{-5/9},g(\bar{p})]\}$ has no points with high probability because its area is of order $n^{-1/2 -5/9} = o(n^{-1})$. Thus conditional on that event, the remaining points are iid uniformly distributed on $\{g < g(\bar{p}) - n^{-5/9}\}$. Let $\cN(T)$ denote the number of points in $T$. For each $p_i \neq \bar{p}$ in $P$, 
\begin{align*}
  \PP[H(p_i) < H(\bar{p}) | \bar{p}, \cN(T) = 0, p_i] &= \PP[\cN(U) \le \cN(V)] \\
&\le \exp(-\Omega(n^{1/9}))
\end{align*}
where $U$ is the union of the darkly shaded regions, and $V$ is the lightly shaded region in Figure~\ref{unit square}.

 The last estimate can be easily obtained using multinomial distribution, or Poisson heuristics. Thus $\bar{p}$-almost surely, we have
\begin{align*}
 \PP[\min_{p_i \neq \bar{p}} H(p_i) < H(\bar{p}) | \bar{p}] &= \PP[\min_{p_i \neq \bar{p}} H(p_i) < H(\bar{p}) | \bar{p}, \cN(T) = 0] + o(1) \\
&\le \sum_{p_i \neq \bar{p}} \PP[ H(p_i) < H(\bar{p})|\bar{p}, \cN(T) = 0] + o(1)\\
&= o(1).
\end{align*}
Taking expectation with respect to $\bar{p}$ concludes the proof.
\end{proof}

\begin{lemma}
For fixed $m$, Let $A_n :=\PP_n[ H(\bar{p}) < m]$. Then $A_n > A_{n+1}(1+\frac{1}{n})^{-1}$, for a fixed constant $c$.
\end{lemma}
\begin{remark}
Actually we are able to show $\PP_n[\rho_\infty(\sigma) < m]$ is monotone decreasing in $n$ (see below), hence we can apply the De Poissonization theorem above directly. However the lemma also yields asymptotics for $H(\tilde{p})$, where $\tilde{p} = \arg \max_p y(p) - x(p)$, for which monotonicity doesn't necessarily hold.
\end{remark}

\begin{proof}
Conditional on $g(\bar{p}_n) = \alpha$, $\bar{p}_n$ is equally likely to be at any point on the union of the line segments $\{g = \alpha\}$, for any $n$. Furthermore, conditional on the exact position of $\bar{p}$, we have by geometric domination
\begin{align*}
\PP_n[ H(\bar{p}) < m | \bar{p}] \ge \PP_{n+1}[ H(\bar{p}) < m | \bar{p}].
\end{align*}
 Therefore,
\begin{align*}
\PP_n[H(\bar{p}) < m| g(\bar{p}) = \alpha] \ge \PP_{n+1}[H(\bar{p}) < m| g(\bar{p}) = \alpha].
\end{align*}
Finally observe that $\PP_n[g(\bar{p}) < \alpha]= (1- (\frac{\sqrt{2}}{2}-\alpha)^2)^n$, which is the probability that there are no points in the top left corner of height $\frac{\sqrt{2}}{2} - \alpha$, hence
\begin{align*}
\PP_n[g(\bar{p}) \in d\alpha] = 2n (1-\alpha) (1-(1-\alpha)^2)^{n-1}.
\end{align*}
This yields $\PP_{n+1}[g(\bar{p}) \in d \alpha] / \PP_n[g(\bar{p}) \in d \alpha] \le 1 + \frac{1}{n}$. Integrating with respect to $\alpha$ gives the result.

\end{proof}

\begin{lemma} \label{H(p)}
Under $\PP^\nu$, the law of the Poisson point process with uniform rate $\nu$ on $[0,1]^2$, $\frac{1}{\nu^{1/2}} H(\bar{p})$ converges weakly to the square root of an exponential distribution with mean $1/2$.
\end{lemma}

\begin{proof}
We first study the distribution of $g(\bar{p})$. Fold the unit square in half along the diagonal $\{x =y\}$, and rotate it so that the hypotenus is contained in the x-axis and the resulting triangle $\Delta$ sits on the upper half plane. This would also overlap the points in the two half-triangles on the original square into the same triangle $\Delta$. Then $g(\bar{p})$ is given by the height of the highest point in the new Poisson point process on $\Delta$, with uniform rate $\mu = 2\nu$. The probabillity that $g(\bar{p}) < \alpha$ is the same as the probability that there are no points in the top similar triangle $\Delta_{\frac{\sqrt{2}}{2} - \alpha}$ of height $\frac{\sqrt{2}}{2} - \alpha$, which is easily calculated to be
\begin{align*}
\PP^{2\nu}[ \cN(\Delta_{\frac{\sqrt{2}}{2} - \alpha}) = 0] = \exp(-2 \nu ( \frac{\sqrt{2}}{2} - \alpha)^2).
\end{align*}
 By a change of variable $\beta = \frac{\sqrt{2}}{2} - \alpha$, we can compute the density of $\sqrt{\nu}\beta$:
\begin{align*}
\PP^{2\nu}(\sqrt{\nu}\beta \in d b) &= -\frac{d}{db} e^{-2b^2} \\
&= 4b e^{-2b^2}.
\end{align*}

Given $g(\bar{p}) = \alpha$, $\bar{p}$ is equally likely to be anywhere on the line $\{g = \alpha\}$, hence the conditional distribution of $H(\bar{p})$ can be calculated by a simple averaging. To avoid such calculation, one could observe that $H(\bar{p})$ is stochastically squeezed between two Poisson random variables $L_{\beta}$ and $U_{\beta}$, with rate $\nu \sqrt{2} \alpha ( 1- \sqrt{2} \alpha) = \nu(1-\sqrt{2} \beta) \sqrt{2} \beta$ and $\nu[(1-\sqrt{2} \beta) \sqrt{2} \beta + \frac{1}{2} (\sqrt{2}\beta)^2]$ respectively. The lower bound $L_{\beta}$ is obtained by combining the two rectangles in the region above and to the left of $\bar{p}$ into one of size $(1-\sqrt{2} \beta) \sqrt{2} \beta$, and the upper bound $U_\alpha$ is obtained by looking at the worst case when $\bar{p}$ is at an edge of the unit square. 

Since $(1-\sqrt{2}\alpha)^2$ is of order $\cO(\frac{\log \nu}{\nu})$ with high probability, the Kolmogorov distance between these two Poisson variables are very small compared to their means.

we can now estimate the moments of $\frac{1}{\sqrt{\nu}} H(\bar{p})$ by $\frac{1}{\sqrt{\nu}} L_{\frac{\sqrt{2}}{2} - \beta}$ and $\frac{1}{\sqrt{\nu}} U_{\frac{\sqrt{2}}{2} - \beta}$, with $\beta$ distributed as above:
\begin{align*}
\E [(\frac{1}{\sqrt{\nu}} L_\beta)^k|\beta] &= \E[ \frac{1}{\sqrt{\nu}}  \frac{L_\beta!}{(L_\beta - k)!}|\beta] + \nu^{-k/2} \E[R_{k-1}(\beta)|\beta]\\
&= \nu^{-k/2} [\sqrt{\nu}(1-\sqrt{2} \beta) \sqrt{2} \sqrt{\nu} \beta]^k + \nu^{-k/2} \E[R_{k-1}(\beta)|\beta]\\
&= [(1-\sqrt{2}\beta) \sqrt{2} \sqrt{\nu} \beta]^k + \nu^{-(k-1)/2} \E[R_{k-1}(\beta)|\beta]
\end{align*}
where $R_{k-1}(\beta)$ is a linear combinations of factorial moments of $L_\beta$ of degree at most $k-1$, hence 
\begin{align*}
\nu^{-(k-1)/2} \E[R_{k-1}(\beta)|\beta] \le \nu^{-(k-1)/2}\cO((\sqrt{\nu} \beta)^{(k-1)}) = \cO(\beta^{(k-1)})
\end{align*}
 The key point is that $\beta$ is concentrated near $\nu^{-1/2}$, hence this term is essentially negligible:
\begin{align*}
\E[(\frac{1}{\sqrt{\nu}} L_\beta)^k] &= \E[[(1-\sqrt{2}\beta) \sqrt{2} \nu^{1/2} \beta]^k ] + C\E \beta^{k-1} \\
&= \int_0^{\frac{\sqrt{2\nu}}{2}} [(1-\sqrt{2}b/\nu^{1/2})\sqrt{2}b]^k 4 b e^{-2b^2} db + o(1)\\
&= \int_0^\infty (\sqrt(2) b)^k e^{-(\sqrt{2}b)^2} d(\sqrt{2}b)^2 + o(1).
\end{align*}
This shows the weak limit of $\frac{1}{\sqrt{\nu}} L_\beta$ has density $4b e^{-2b^2}db 1_{\{b \ge 0\}}$. This is the density of the square root of an exponential random variable $X$ with mean $1/2$. Similarly, $\frac{1}{\sqrt{\nu}} U_\beta$ weakly converges to $\sqrt{X}$ as well. Thus $H(\bar{p})^k$ converges to the same thing as
\begin{align*}
\E (\frac{1}{\sqrt{\nu}}L_\beta)^k \le \E (\frac{1}{\sqrt{\nu}}H(\bar{p}))^k \le \E (\frac{1}{\sqrt{\nu}}U_\beta)^k. 
\end{align*}
Note that if we define $\tilde{p} = \arg \max_{p: y(p) > x(p)} d(p,\{x = y\})$, then $\frac{1}{\sqrt{\nu}}H(\tilde{p})$ converges to the square root of an exponential random variable with mean $1$. This corresponds to the following one-sided Spearman's uniform metric on $S_n$:
\begin{align*}
\tilde{\rho}_\infty(\sigma) = \max_{i \le n} (i - \sigma(i)).
\end{align*}
\end{proof}

\begin{corollary}
$\frac{1}{\sqrt{n}}[n-\rho_\infty(\sigma)]$ converges weakly to the square root of an exponential random variable with mean $1/2$.
\end{corollary}
Note that we are not able to show convergence in moments, because the dePoissonization procedure requires a bounded test function.

Next we examine relation of $\rho_\infty$ with $\rho_q$ for finite $q$ as well as the lengths of the longest increasing and decreasing subsequences, denoted $I(\sigma)$ and $D(\sigma)$ respectively. First we need a lemma
\begin{lemma}
Fix $m_1,m_2,\ldots, m_\infty \ge 0$, consider the events $A_j = \{\rho_j(\sigma) < m_j\}$, $1 \le j \le \infty$. Let $B \subseteq \N \cup \{\infty\}$ be any subset. Then $\PP_n[\cap_{j \in B} A_j]$ is a weakly decreasing sequence in $n$.
\end{lemma}

\begin{proof}
For any $\sigma \in S_n$, we construct a random element $\tau \in S_{n+1}$ by the Chinese restaurant process, i.e., either letting $n+1$ to be a fixed point of $\tau$ or inserting $n+1$ into an existing cycle randomly. In other words, with probability $\frac{1}{n+1}$, we let $\tau(n+1) = n+1$ and $\tau(j) = \sigma(j)$ for $ j \in [n]$ and with the remaining probability we choose an element $k \in [n]$ uniformly at random, and modify $\sigma$ by letting $\tau(k) = n+1$, $\tau(n+1) = \sigma(k)$ and $\tau(j) = \sigma(j)$ for all other $j \in [n] \setminus \{k\}$. For each $\rho_q$, this either introduces a new term $|n+1 - (n+1)|^q$, or replaces the term $|k - \sigma(k)|^q$ by two new terms $|k - (n+1)|^q + |n+1 - \sigma(k)|^q \ge |k - \sigma(k)|^q$. Hence $\rho_q(\tau) \ge \rho_q(\sigma)$. Thus $\rho_j(\tau) < m_j$ for all $j \in B$ implies $\rho_q(\sigma) < m_q$ for all $j \in B$, and the assertion follows.
\end{proof} 
\begin{remark}
The Chinese restaurant coupling of $S_n$ with $S_{n+1}$ does not give monotonicity of the distribution of $I(\sigma)$ and $D(\sigma)$, as the process of inserting $n+1$ into a cycle can destroy an increasing or decreasing subsequence.
\end{remark}

\begin{theorem} \label{rho_q}
Fix $k \in \N$. Then
$\sqrt{1}{\sqrt{n}}[n-\rho_\infty(\sigma)]$ is asymptotically weakly independent
from the following random vector $(\tilde{\rho}_1,\ldots, \tilde{\rho}_k)$, where $\tilde{\rho}_q  = \frac{1}{\sqrt{\var \rho_q}} (\rho_q - \E \rho_q)$. 
\end{theorem}
\begin{proof}
For simplicity, we only consider the case of a single $q$; the more general vector version follows the same reasoning. 

Using the monotonicity lemma above, and the general de-Poissonization theorem of Johansson, it suffices to prove asymptotic independence under the Poissonized Hammersley's point process with rate $\nu$, as $\nu$ goes to infinity. To be precise, for an instance $P$ of the Poisson point process on $[0,1]^2$, define
\begin{align*}
\rho_q(P) = \sum_{p \in P} |X(p) - Y(p)|^q.
\end{align*}
and similarly, $\rho_\infty(P) = \max_{p \in P} |X(p) - Y(p)|$ as before.

It is known that Spearman's footrule has mean of order $n^2$ and variance $n^3$, and Spearman's rho has mean of order $n^3$ and variance $n^5$. In general it's not too hard to show $\rho_q$ has mean of order $n^{q+1}$ and variance of order $n^{2q + 1}$. 

Then by Poissonization we know $\E^\nu \rho_q = \mu_q \nu^{q+1} + \cO(\nu^q)$, and $\var^\nu \rho_q = v_q \nu^{2q +1} + \cO(\nu^{2q})$, where $\mu_q, v_q > 0$, and that $\rho_q$ converges to a standard normal variable $N_q$ upon normalization. We need to show that conditional on $\rho_\infty \approx m_\infty$ a typical value, $\rho_q$ still converges to a standard normal variable that's close to $N_q$.

So conditional on $\frac{1}{\sqrt{\nu}}[\nu - \rho_\infty] > \beta$, we have $g(\bar{p}) > \frac{\sqrt{2}}{2} (1-\beta) + o(\nu^{-1/2})$ and $\bar{p} = p^*$ with high probability. Thus $\cN(\Delta_\beta) = 0$ where $\Delta_\beta = \{p: g(p) \ge g(\bar{p})\}$ and we have a Poisson point process $U_1$ on $[0,1]^2$ with uniform rate $\nu$ on the region $R_\beta:=\{g \le g(\bar{p})\}$, and rate $0$ on its complement $\Delta_\beta$. 

Now consider a new point process $U_2$ on $[0,1]^2$ coupled to $U_1$ so that $U_2$ equals $U_1$ on $R_\beta$ and in the complement of $R_\beta$, $U_2$ is given by an independent Poisson point process with rate $\nu$. Thus $U_2$ is a point process with uniform rate $\nu$ on the entire square. With high probability, we have $\beta \le \nu^{-1/2} \log \nu$, and also with high probability under $U_2$, $\cN(\Delta_\beta) < (\log \nu)^3$. Therefore
\begin{align*}
|\rho_q(U_2) - \rho_q(U_1)| \le \cO(\nu^q (\log \nu)^3) << \sqrt{\var^\nu \rho_q}. 
\end{align*}
  Therefore we have shown that for all $\beta$, there is a coupling under which
  \begin{align*}
  \lim_\nu \frac{ \cL(\rho_q | \rho_\infty > \nu - \beta \nu^{1/2}) - \rho_q}{\var^\nu \rho_q} = 0
  \end{align*}
  in probability. In particular this implies $\frac{\cL(\rho_q |\rho_\infty > \nu -\beta \nu^{1/2} ) - \mu_q}{v_q}$ converges weakly to a standard Gaussian for all $\beta$. Hence $\rho_q$ is asymptotically independent of $\rho_\infty$.
  \end{proof}
   
   Next we consider the correlation between $\rho_\infty$ and the first $k$ rows and columns of the RSK algorithm output. Let $I_1(\sigma),\ldots, I_k(\sigma)$ be the lengths of the first $k$ rows of the Young diagram obtained from the RSK algorithm applied to $\sigma$. Similarly let $D_1(\sigma),\ldots, D_k(\sigma)$ be the lengths of the first $k$ columns. Curtis Greene has shown that $\sum_{i=1}^k I_i(\sigma) = \max |\cup_{i=1}^k S^I_i|$ where $S^I_i$ are increasing subsequences of $\sigma$. Similarly $D_i(\sigma) = \max |\cup_{i=1}^k S^D_i|$ where $S^D_i$ are decreasing subsequences of $\sigma$. Thus for example, if $k=1$, $I_1(\sigma)$ is the length of the longest increasing subsequence, which was known long before Greene's result.
   
   \begin{theorem}
   $\rho_\infty$ is asymptotically independent of the random vector $(I_1,\ldots, I_k, D_1,\ldots, D_k)$ in the following sense. Let $\tilde{\rho}_\infty = \frac{1}{\sqrt{n}}(n - \rho_\infty)$ be the normalized version of $\rho_\infty$, and $\tilde{I}_j  = \frac{1}{n^{1/6}}(I_j - \sqrt{2n})$, $\tilde{D}_j = \frac{1}{n^{1/6}}(D_j - \sqrt{2n})$ the normalized version of $I_j$ and $D_j$ respectively. Then for fixed $m^I_1, \ldots, m^I_k, m^D_1,\ldots, m^D_k, m_\infty$,
   \begin{align*}
  \lim_{n \to \infty} \PP_n[\tilde{\rho}_\infty < m_\rho, \tilde{I}_j > m^I_j, \tilde{D}_j > m^D_j, j \in [k]] - \PP_n[\tilde{\rho}_\infty < m_\rho] \PP_n[\tilde{I}_j > m^I_j, \tilde{D}_j > m^D_j, j \in [k]] = 0.
\end{align*}
\end{theorem}
\begin{proof}
Let $Q_n = \PP_n[\tilde{\rho}_\infty > m_\rho, \tilde{I}_j > m^I_j, \tilde{D}_j > m^D_j, j \in [k]]$. We first show that 
\begin{align} \label{pseudo monotone Q_n}
 Q_n < (1+ \frac{c\log n}{n})Q_{n+1}.
\end{align}
 We consider the following coupling between a uniformly random element $\sigma \in S_n$ and $\tau \in S_{n+1}$. Let $\sigma$ be given by an instance $P$ of the Hammersley's point process, then $\tau$ is given by adding another point uniformly in $[0,1]^2$ and independently from the points in $P$. Using Greene's interpretation, it is evident that $I_j(\tau) \ge I_j(\sigma)$ and $D_j(\tau) \ge D_j(\sigma)$ for all $j$. Furthermore given $\sigma$, $\tilde{\rho}_\infty(\tau) > \tilde{\rho}_\infty(\sigma)$ only if the new point $p_{n+1}$ lands in the union of two triangular regions $\Delta_{g(\bar{P})}$. By a direct computation, we see that $\frac{\sqrt{2}}{2} - g(\bar{P})$ is dominated by the square root of a geometric random variable with mean of order $n^{-1/2}$ (see also Lemma~\ref{H(p)}). Hence with high probability ($1- \cO(n^{-c})$), $\text{vol}(\Delta_{g(\bar{P})}) \le \frac{c \log n}{n}$. Therefore 
\begin{align*}
\PP_n[ \tilde{\rho}_\infty(\tau) > \tilde{\rho}_\infty(\sigma) ] \le \frac{c \log n}{n} + \cO(n^{-c}).
\end{align*}
These two considerations imply \eqref{pseudo monotone Q_n}. Thus combined with the De-Poissonization Corollary~\ref{De-Poisson}, it suffices to show 
\begin{align*}
\lim_{\nu \to \infty} \PP^\nu[\tilde{\rho}_\infty < m_\rho, \tilde{I}_j < m^I_j, \tilde{D}_j < m^D_j, j \in [k]] - \PP^\nu[  \tilde{\rho}_\infty < m_\rho] \PP^\nu [\tilde{I}_j < m^I_j, \tilde{D}_j < m^D_j, j \in [k]] = 0
\end{align*}
where $\tilde{I}_j$ is defined to be the normalized maximum size of the union of $j$ increasing subsequences in the Hammersley's square, with the normalization scale $n$ replaced by $\nu$. The other variables are defined similarly. The idea is similar to the proof of Theorem~\ref{rho_q}; here we provide a bit more detail.

 We first condition on $G(\bar{p}) := 1 - \sqrt{2} g(\bar{p})$. With high probability $G(\bar{p}) \le \nu^{-1/2} \log \nu$. Conditional on $G(\bar{p})$, $H(\bar{p})$
is a mixture of Poisson random variables with rates bounded in the interval $[\nu G(\bar{p}) -c \log \nu,\nu G(\bar{p}) +c \log \nu]$, for some constant $c$. Since Poisson distribution of rate $\sqrt{\nu}$ has mean $\sqrt{\nu}$ and variance of order $\nu^{1/4}$, and in fact behaves like a Gaussian near its mean, we have 
\begin{align*}
\PP^\nu[ G(\bar{p}) \in[ \beta - \frac{\log \nu}{\sqrt{\nu}}, \beta + \frac{\log \nu}{\sqrt{\nu}}] | \frac{H(\bar{p})}{\sqrt{\nu}} = \beta] = 1 + o(1).
\end{align*}
 Therefore if we let $\Omega= \Omega(m^I_j, m^D_j):= \{\tilde{I}_j > m^I_j, \tilde{D}_j > m^D_j, j \in [k]\}$, then 
\begin{align*}
\PP^\nu( \Omega \cap \{G(\bar{p}) \le \beta -\frac{\log \nu}{\sqrt{\nu}}\}) +o(1) &\le \PP^\nu(\Omega \cap \{H(\bar{p}) \le \beta\}) \\
&\le \PP^\nu(\Omega \cap \{G(\bar{p}) \le \beta +\frac{\log \nu}{\sqrt{\nu}}\}) + o(1).
\end{align*}
Thus if one can show $\Omega$ and $G(\bar{p})$ are asymptotically independent, then the two sides of the inequalities above would be asymptotically equal (because $G(\bar{p})$ has continuous distribution function), which would also imply the asymptotic independence of $H(\bar{p})$ and $\Omega$. Since we know $H(\bar{p})= \nu - \rho_\infty(P)$ with high probability, this would further imply the asymptotic independence of $\Omega$ and $\tilde{\rho}_\infty(P)$. So it remains to show $\Omega$ and $G(\bar{p})$ are asymptotically independent.

  Conditioning on $G(\bar{p}) \le \beta$ is equivalent to conditioning on $\cN(\Delta_{\frac{\sqrt{2}}{2 \sqrt{\nu}} \beta}) = 0$, i.e., there are no points in the top left and bottom right corner isosceles right triangles of leg length $\frac{\beta}{\sqrt{\nu}}$. The point process on the complement of $\Delta = \Delta_{\frac{\sqrt{2}}{2 \sqrt{\nu}} \beta}$ has a uniform Poisson rate of $\nu$. Call this point process $U_1$. Now if we construct $U_2$ on $[0,1]^2$ based on $U_1$ by adding an independent point process on $\Delta$, then $U_2$ is simply a PPP of uniform rate $\nu$ on the entire square. Since adding points would only increase the values of $I_j$, $D_j$, for any $j$, we get the following bound:
  \begin{align*}
  \PP^\nu[ \Omega | G(\bar{p}) \le \beta] \le \PP^\nu[\Omega].
  \end{align*}
  Next suppose an increasing subsequence $S$ contains a point in $\Delta$, without loss of generality, say $S \subset \{ x \le y\}$  the upper left component of $\Delta$. Then the remaining portion of $S$ is contained in the region $\Gamma_\beta :=\Gamma^x \cup \Gamma^y:=\{x \le \frac{\beta}{\sqrt{\nu}}\} \cup \{y \ge 1 -\frac{\beta}{\sqrt{\nu}}\}$. But since $\PP^\nu[\cN(\Gamma^x) \ge \frac{\log \nu}{\sqrt{\nu}}] = o(1)$ and  $\PP^\nu[\cN(\Gamma^y) \ge \frac{\log \nu}{\sqrt{\nu}}] = o(1)$.  Hence 
  \begin{align} \label{no Gamma shape}
  \PP^\nu[|S| \ge \frac{\log \nu}{\nu^{1/4}} | S \cap \Delta \neq \emptyset] = o(1). 
  \end{align}
  Now consider the events $B_1 := \cup_{j \le k}\{I_j(U_1) \neq I_j(U_2)\}$ and $B_2:= \cup_{j \le k}\{D_j(U_2) - D_j(U_1) \ge \log \nu\}$. If we can show 
  \begin{align} \label{bad set}
  \PP^\nu(B_1 \cup B_2) = o(1),
 \end{align} 
  then the vector 
  \begin{align*}
  \frac{1}{\nu^{1/6}}(I_1(U_1) - I_1(U_2),\ldots, I_k(U_1)-I_k(U_2), D_1(U_1) - D_1(U_2), \ldots, D_k(U_1) - D_k(U_2))
  \end{align*}
   converges to $0$ in probability, which would imply the conditional distribution $\cL(\tilde{I}_j, \tilde{D}_j, j \in [k] | G(\bar{p}) \le \beta)$ converges weakly to the unconditional $\cL(\tilde{I}_j, \tilde{D}_j, j \in [k])$, from which asymptotic independence of $\Omega$ and $G(\bar{p})$ is immediate.

   $B_1$ implies there exist some $p \in \Delta$, and an increasing subsequence $S$ containing $p$ such that $|S| \ge \min_{j \le k} |I_j(U_2)|$. By \eqref{no Gamma shape}, this event has probability $o(1)$. On the other hand, $B_2$ implies the number of points in $U_2 \setminus U_1$, i.e., the extra points in $\Delta$, is more than $\log \nu$. But this also has vanishing probability due to the fact that $\text{vol}(\Delta) = o(\frac{\log \nu}{\nu})$ with high probability, and the Poisson mean of $U_2$ on $\Delta$ is proportional to $\nu \text{vol}(\Delta)$. 
   
  \end{proof}

\begin{theorem}
 Under uniform measures on $S_n$, Spearman's uniform metric $\rho_\infty(\sigma)$ is asymptotically independent from any sequence of class functions $f_n$ on $S_n$ that has a weak limit. 
\end{theorem}
\begin{remark}
We use the word metric loosely here to mean a univariate function on $S_n$ given by the distance between its argument and the identity element. 
\end{remark}
\begin{proof}
We will show that for almost all $\lambda \vdash n$, conditional on the cycle type of $\sigma$ being $\lambda$, $\rho_\infty(\sigma)$ has the same law as the unconditional law with high probability. 

First note that the same strategy used to prove independence of $\rho_q(\sigma)$ and $f(\sigma)$ would not work here since changing even one bracket of $\sigma$ can change the value of $\rho_q$ a lot. But the crucial observation here is that the bad bracketing positions occur with vanishing probability. In order to get a good estimate of this probability, we first need to symmetrize the record map as follows. 

Recall the record map $r$ does the following to a permutation $\sigma \in S_n$ 
\begin{enumerate}
\item it arranges the elements in its cycles cyclically so that the biggest element appears in the first position, and then arranges the cycles by the increasing order of their first elements. The end result is still $\sigma$.  
\item it removes all the brackets and view the resulting $n$-sequence as the second row in the 2-line notation of a permutation. 
\end{enumerate}
It is well-known that $r:S_n \to S_n$ is a bijection. Now for $\tau \in S_n$, define 
\begin{align*}
r_\tau(\sigma) = \tau^{-1} r(\tau \sigma \tau^{-1}) \tau.
\end{align*}
In words, this means we perform the record map on $\sigma$, with the natural ordering on $[n]$ replaced by the new ordering defined by $i <_\tau j$ if $\tau(i) < \tau(j)$. 

Finally recall that $S_n^\lambda$ denotes the set of permutations with cycle structure $\lambda$, and the map $\varphi_\lambda:S_n \to S_n^\lambda$ takes each $\sigma$ to $(\sigma(1),\ldots,\sigma(\lambda_1))(\sigma(\lambda_1+1),\ldots,\sigma(\lambda_1 + \lambda_2))\ldots (\sigma(n-\lambda_l+1),\ldots, \sigma(n))$. $\varphi_\lambda$ takes the second row of $\sigma$ under 2-line notation and inserts brackets into it to arrive at a new permutation in the cycle notation. 

For each $\tau, \tau' \in S_n$, and $\sigma$ uniform in $S_n$, $\varphi_\lambda \circ r_\tau' \circ \varphi_{(n)} \circ r_\tau(\sigma)$ is uniform on $S_n^\lambda$. Therefore if we take $\tau$ uniformly random and independent from $\sigma$, the result is still uniform on $S_n^\lambda$. 

Next observe that $\varphi_{(n)} \circ r_\tau(\sigma)$ picks uniformly a symbol $i_j$ from each cycle $C_j$ of $\sigma$ and changes the value $\sigma(i_j)$; to us it's not important what the modified values are. For instance, if $\sigma = (124)(536)$, then with probability $(1/3)^2$, $r_\tau(\sigma)$ picks $2$ from the first cycle and $6$ from the second cycle, and changes the value of $\sigma(2)$ to $5$ and $\sigma(6)$ to $4$.

Let $i^*(\sigma) = \arg \max |i - \sigma(i)|$ taken to be the smallest such if there are more than one maximizers.

We will first show that these modifications do not affect $\rho_\infty(\sigma)$ with high probability, i.e.,
\begin{align*}
\PP[\rho_\infty(\sigma) \neq \rho_\infty(\varphi_{(n)} \circ r_\tau (\sigma))] = o(1).
\end{align*}
First we need a lemma
\begin{lemma}
Under the uniform measure, $i^*(\sigma)$ resides in a cycle of length at least $(\log n)^4$ with high probability.
\end{lemma}
\begin{proof} 
Let $\alpha_i(\sigma)$ be the number of $i$-cycles in $\sigma$. By a result of Arratia, Barbour, and Tavare, we have the following approximate coupling result
\begin{align*}
\| \cL(\alpha_1, \ldots, \alpha_{(\log n)^4}) - \otimes_{1 \le i \le (\log n)^4} \mu_{1/i} \|_{\text{TV}} = o(1)
\end{align*}
where $\mu_{1/i}$ are independent Poisson distribution of rate $1/i$. 

We will call a cycle small if its length is less than $(\log n)^4$. Using Chebyshev inequality, we can show that the union of all small cycles has size less than $(\log n)^9$ with high probability:
\begin{align*}
\PP[ \sum_{i=1}^{(\log n)^4|} i \alpha_i > (\log n)^9] &\le \E[\exp(t \sum_{i=1}^{(\log n)^4} i\mu_{1/i})] / \exp(t (\log n)^9) + o(1)\\
 &= \prod_{i=1}^{(\log n)^4} e^{\frac{1}{i}(e^{it} - 1)} / \exp(t (\log n)^9) + o(1)\\
 &= \exp( \sum_{i=1}^{(\log n)^4} \frac{1}{i}(e^{it} -1) - t (\log n)^9) + o(1)
 \end{align*}
 so if we choose $t  = (\log n)^{-5}$, we can approximate $e^{it} -1$ by $it$ for $i \le (\log n)^4$ and the last quantity is bounded by $\exp(- (\log n)^4) + o(1) = o(1)$.
 
Now fix a typical cycle type $\lambda$ as above. Let $C_0= C_0(\lambda)$ be the union of all the small parts of $\lambda$. Let $\sigma$ be uniformly chosen from $S_n^\lambda$. We can estimate the distribution of $\rho_\infty$ on $\sigma$ restricted to $C_0$ as follows:
\begin{align*}
\PP_{S_n^\lambda}[\rho_\infty(\sigma|_{C_0}) < n - n^{8/9}] \ge (1- 2n^{-1/9})^|C_0| =  1 + o(1)
\end{align*}
for $|C_0| \le (\log n)^9$. This uses the fact that if the symbols of $\sigma|_{C_0}$ are all contained in $[n^{8/9}, n-n^{8/9}]$, then $\rho_\infty(\sigma|_{C_0}) < n- n^{8/9}$. 

If we can show that $\rho_\infty (\sigma|_{[n] \setminus C_0}) > n - n^{7/9}$ with high probability, then $i^*(\sigma)$ must reside in one of the big cycles. In fact we can replace $[n] \setminus C_0$ by just the longest cycle $C_1$. It is well-known that the normalized longest cycle length converges to the Dickman distribution which is strictly positive on $(0,1)$, hence $|C_1| > n^{8/9}$ with high probability. 

We will consider consecutive pairs of adjacent symbols in $\sigma|_{C_1}$. One way of choosing a uniformly random element from $S_n^\lambda$ is by choosing a uniformly random string $s=(s_1,s_2,\ldots, s_n)$ of length $n$ consisting of nonrepeating elements in $[n]$, and then imposing the appropriate brakcets on $s$ to have the desired cycle structure $\lambda$.  Alternatively one can lay down the brackets first and then fill in the entries $s_1,s_2,\ldots$ sequentially by sampling uniformly without replacement from $[n]$. We adopt the second point of view and let the left-most bracket corresponds to the longest cycle. Consider the pairs of elements $(s_{2k-1},s_{2k})$ for $k=1,\ldots, n^{5/9}$. For each $k$, the probability that $|s_{2k-1}- s_{2k}| > n- n^{7/9}$ given the previous $s_j$'s is bounded below by $\frac{(n^{7/9} - n^{5/9})^2}{n^2} = \Omega(n^{-4/9})$, which comes from the worst case when all the previous $(s_{2j-1},s_{2j})$ pairs land in $[n^{7/9}] \cup [n-n^{7/9},n]$. Therefore the probability that $|s_{2k-1} - s_{2k}| < n - n^{7/9}$ for all $k \le n^{5/9}$ is bounded above by $(1-n^{-4/9})^{n^{5/9}} = o(1)$. This concludes the proof of the lemma.
\end{proof}  

Using the lemma, we see that for asymptotically almost all $\sigma$, $\rho_\infty(\varphi_{(n)} \circ r_\tau(\sigma)) = \rho_\infty(\sigma)$, because if $i^*(\sigma)$ lies in a cycle of length $k > (\log n)^9$, then the probability that $i^*(\sigma)$ is chosen in the bracket removing process of $r_\tau$ would be $1/k$. Next we see that $r_{\tau'}$ essentially permutes the sequence $r_\tau(\sigma)$ cyclically in a uniformly random way. Therefore for a cycle structure $\lambda$ of at most say $(\log n)^2$ cycles, which is with high probability, the chance that $\varphi_\lambda \circ r_{\tau'}(\sigma)$ modifies a particular value $\sigma(i^*)$ is bounded above by $(\log n)^2/n$. Thus we have shown the following:

for almost all $\lambda$ under the uniform measure on $S_n$, and almost all $\sigma \in S_n$,
\begin{align*}
\PP[ \rho_\infty(\varphi_\lambda \circ r_{\tau'} \circ \varphi_{(n)} \circ r_\tau (\sigma)) \neq \rho_\infty(\sigma)] = o(1).
\end{align*}
 Since $\varphi_\lambda \circ r_{\tau'} \circ \varphi_{(n)} \circ r_\tau: S_n \to S_n^\lambda$ pushes forward the uniform measure to the uniform measure, we have
 \begin{align*}
 \PP_{S_n^\lambda}[ \rho_\infty(\sigma) < \alpha] &= \PP_{S_n} [\rho_\infty(\varphi_\lambda \circ r_{\tau'} \circ \varphi_{(n)} \circ r_\tau(\sigma)) < \alpha] \\
 &= \PP_{S_n}[\rho_\infty(\sigma) < \alpha] + o(1).
 \end{align*}
Finally by Baye's rule,
\begin{align*}
\PP[ \rho_\infty(\sigma) < \alpha, f(\sigma) < \beta] &= \E[ \PP[\rho_\infty(\sigma)< \alpha, f(\sigma) < \beta| \lambda(\sigma)]] \\
&= \E[ \PP[\rho_\infty(\sigma) < \alpha|\lambda] 1_{\{f(\lambda) < \beta\}} + o(1)] + o(1)\\
&= \PP[ \rho_\infty(\sigma) < \alpha] \PP[f(\sigma) < \beta] + o(1).
\end{align*}

\end{proof}

\section{Acknowledgement}
I am grateful to Persi Diaconis for directing my interest towards metric structures on finite groups, and suggesting numerous challenging research problems along the way. I would also like to thank Colin Mallows \cite{CM} for suggesting the fruitful investigation of asymptotic independence between Cayley's distance and Kendall's tau distance. Throughout the writing, I have benefited greatly from discussion with other graduate students at the Stanford math department.

\end{document}